\documentclass[12pt, reqno]{amsart}
\usepackage{amsmath, amsthm, amscd, amsfonts, amssymb, graphicx, color}
\usepackage[bookmarksnumbered, colorlinks, plainpages]{hyperref}
\hypersetup{colorlinks=true,linkcolor=red, anchorcolor=green, citecolor=cyan, urlcolor=red, filecolor=magenta, pdftoolbar=true}
\input{mathrsfs.sty}

\newtheorem{theorem}{Theorem}[section]
\newtheorem{lemma}[theorem]{Lemma}

\newtheorem{cor}[theorem]{Corollary}
\theoremstyle{definition}

\theoremstyle{remark}
\newtheorem{remark}[theorem]{\bf{Remark}}
\numberwithin{equation}{section}
\begin{document}

\title [Euclidean operator radius inequalities of a pair operators] {\small {Euclidean operator radius inequalities of a pair of bounded linear operators and their applications}}

\author[S. Jana P. Bhunia and K. Paul] {Suvendu Jana$^1$, Pintu Bhunia$^2$ \MakeLowercase{and} Kallol Paul$^2$}

\address{$^1$ Department of Mathematics, Mahisadal Girls College, Purba Medinipur 721628, West Bengal, India}
\email{janasuva8@gmail.com}

\address{$^2$ Department of Mathematics, Jadavpur University, Kolkata 700032, West Bengal, India}
\email{pintubhunia5206@gmail.com}
\email{kalloldada@gmail.com}

\thanks{Pintu Bhunia would like to thank UGC, Govt. of India for the financial support in the form of SRF under the mentorship of Prof. Kallol Paul. }

\renewcommand{\subjclassname}{\textup{2020} Mathematics Subject Classification}

\maketitle

\begin{abstract}
	We obtain several sharp lower and upper bounds for the Euclidean operator radius of a pair  of bounded linear operators defined on a complex Hilbert space. As applications of these bounds we deduce a chain of  new bounds for the classical numerical radius of a bounded linear operator which improve on the existing ones. In particular, we prove that for a bounded linear operator $A,$
	 \[\frac{1}{4} \|A^*A+AA^*\|+\frac{\mu}{2}\max \{\|\Re(A)\|,\|\Im(A)\|\}
	\leq  w^2(A) \, \leq \, w^2( |\Re(A)| +i |\Im(A)|),\]
	where  $\mu= \big| \|\Re(A)+\Im(A)\|-\|\Re(A)-\Im(A)\|\big|.$ This improve the existing upper and lower bounds  of the numerical radius, namely, 
	 \[ \frac14 \|A^*A+AA^*\|\leq w^2(A) \leq \frac12 \|A^*A+AA^*\|. \]
\end{abstract}

\section{Introduction}

Let $\mathscr{H}$ be a complex Hilbert space with inner product $\langle \cdot,\cdot \rangle $ and the norm $\|\cdot\|$ induced by the inner product. Let $ \mathbb{B}(\mathscr{H})$ denote the $C^*$-algebra of all bounded linear operators on $\mathscr{H}.$  For $A\in \mathbb{B}(\mathscr{H}),$ $A^*$ denotes the adjoint of $A$ and  $|A|=({A^*A})^{\frac{1}{2}}$ denotes the positive square root of $A.$  The real part and imaginary part of $A,$ denoted by  $ \Re (A)$ and $\Im(A),$    are defined as $\Re(A)=\frac{1}{2}(A+A^*)$ and $\Im(A)=\frac{1}{2\rm i}(A-A^*)$ respectively. The numerical range of $A$, denoted by $W(A)$, is defined as $W(A)=\left \{\langle Ax,x \rangle: x\in \mathscr{H}, \|x\|=1 \right \}.$
We denote by $\|A\|$, $ c(A) $ and $w(A)$ the operator norm, the Crawford number and the numerical radius of $A$, respectively. Note  that $$c(A)=\inf \left \{|\langle Ax,x \rangle|: x\in \mathscr{H}, \|x\|=1 \right \}$$ and $$w(A)=\sup \left \{|\langle Ax,x \rangle|: x\in \mathscr{H}, \|x\|=1 \right \}.$$ 
It is well known that the numerical radius $ w(\cdot)$ defines a norm on $\mathbb{B}(\mathscr{H})$ and is equivalent to the operator norm $\|\cdot\|$. In fact, the following double inequality holds:
\begin{eqnarray}\label{eqv}
\frac{1}{2} \|A\|\leq w({A})\leq\|A\|.
\end{eqnarray}
The inequalities in (\ref{eqv}) are sharp. The first inequality becomes equality if $A^2=0$, and the second one turns into equality if $A$ is normal. Kittaneh \cite{E} improved the inequalities in (\ref{eqv}) by establishing that 
\begin{eqnarray}\label{k5}
\frac{1}{4}\|A^*A+A{A}^*\|\leq w^2({A})\leq\frac{1}{2}\|A^*A+A{A}^*\|.
\label{d}\end{eqnarray}
For further improvements of \eqref{eqv} and \eqref{k5} we refer the interested readers to \cite{MIA20,AM21,LAA21,RIM21,BSM21,STU03,STU07}.
Let $B,C\in \mathbb{B}(\mathscr{H})$. The Euclidean operator radius of $B$ and $C$, denoted by $w_e(B,C),$ is defined as $$w_e(B,C)= \sup \left \{ \sqrt{|\langle B x,x\rangle|^2+|\langle C x,x\rangle|^2} : x\in \mathscr{H}, \|x\|=1 \right \}.$$ 
Following \cite{P}, $w_e(\,,\,): \mathbb{B}^2(\mathscr{H})\to [0,\infty]$ is a norm that satisfies the inequality
 \begin{eqnarray}
\frac{\sqrt{2}}{4}\|B^*B+C^*C\|^{\frac{1}{2}}\leq w_e(B,C)\leq\|B^*B+C^*C\|^{\frac{1}{2}}.
\label{eqn1}\end{eqnarray}
The constants $\frac{\sqrt{2}}{4}$ and $1$ are best possible in \eqref{eqn1}.
 If $B$ and $C$ are self-adjoint operators, then $(\ref{eqn1})$ becomes 
 \begin{eqnarray}
\frac{\sqrt{2}}{4}\|B^2+C^2\|^{\frac{1}{2}}\leq w_e(B,C)\leq\|B^2+C^2\|^{\frac{1}{2}}.
\label{eqn2}\end{eqnarray}
We note that for self-adjoint operators $B$ and $C$, $w_e(B,C)=w(B+i C),$ its proof follows easily from the definition of $w_e(B,C)$.
In \cite[Th. 1]{D}, Dragomir proved that if  $B,C\in \mathbb{B}(\mathscr{H})$, then
\begin{eqnarray}\label{D06}
\frac12 w(B^2+C^2)	\leq w^2_e(B,C)\leq\|B^*B+C^*C\|,
\end{eqnarray}
 where the constant $\frac12$ is best possible in the sense that it cannot be replaced by a larger constant.  For further extension of Euclidean operator radius and related inequalities we refer to \cite{SMS, SAH21}.

In this article, we develop several lower and upper bounds for the Euclidean operator radius of a pair $(B,C)$ of bounded linear operators defined on $\mathscr{H}$ which generalize and improve on the existing bounds \eqref{eqn1} and \eqref{D06}. Applying the Euclidean operator radius bounds obtained here, we obtain new lower and upper bounds for the numerical radius of a bounded linear operator which improve on the bounds in \eqref{k5}.

\section{Main Results}

To proof our first theorem we need the following lemma.\begin{lemma}\cite{a1}(Cauchy-Schwarz inequality)
 If $A\in\mathbb{B}(\mathscr{H})$ and $ 0\leq\alpha\leq1$, then $$ |\langle Ax,y\rangle|^2\leq\langle|A|^{2\alpha} x,x\rangle\langle|A^*|^{2(1-\alpha)}y,y\rangle$$ for all $x,y\in\mathscr{H}$.

\label{lem1}\end{lemma}

\begin{theorem}\label{th1} Let  $B,C \in\mathbb{B}(\mathscr{H})$, then 
	\begin{eqnarray*}
&& \frac{1}{2} w(B^2+C^2)+\frac{1}{2}\max \{w(B),w(C)\}\big| w(B+C)-w(B-C)\big|\\
&& \leq w_e^2(B,C)\leq  \min \left \{ w(|B|+i|C|)w(|B^*|+i|C^*|), w(|B|+i|C^*|)w(|B^*|+i|C|) \right \}.
\end{eqnarray*}
\end{theorem}
\begin{proof}
Let $x$ be an unit vector in $\mathscr{H}$. Then we have \begin{eqnarray*}
|\left\langle Bx,x\right\rangle|^2+|\left\langle Cx,x\right\rangle|^2&\geq& \frac{1}{2}\left(|\left\langle Bx,x\right\rangle|+|\left\langle Cx,x\right\rangle|\right)^2\\&\geq&\frac{1}{2}\left(|\left\langle Bx,x\right\rangle \pm \left\langle Cx,x\right\rangle|\right)^2\\&=&\frac{1}{2}|\left\langle (B\pm C)x,x\right\rangle|^2.
\end{eqnarray*}
Taking supremum over all $x$ in $\mathscr{H}$, $\|x\|=1$, we get \begin{eqnarray}
w_e^2(B,C)\geq\frac{1}{2} w^2(B\pm C).
\label{eq1}\end{eqnarray}
Therefore, it follows from the inequalities in (\ref{eq1}) that
 \begin{eqnarray*}
 w_e^2(B,C)&\geq&\frac{1}{2}\max \{w^2(B+C), w^2(B-C)\}\\&=& \frac{w^2(B+C)+w^2(B-C)}{4}+\frac{\big|w^2(B+C)-w^2(B-C)\big|}{4}\\&\geq& \frac{w((B+C)^2)+w((B-C)^2)}{4}\\
 && +(w(B+C)+w(B-C))\frac{\big|w(B+C)-w(B-C)\big|}{4}\\&\geq& \frac{w((B+C)^2+(B-C)^2)}{4}\\
 && +(w(B+C)+(B-C))\frac{\big|w(B+C)-w(B-C)\big|}{4}.	 \end{eqnarray*}
Therefore,
\begin{eqnarray}\label{p1}
	w_e^2(B,C)&\geq&\frac{w(B^2+C^2)}{2}+\frac{w(B)}{2}\big|w(B+C)-w(B-C)\big|.
\end{eqnarray}
Interchanging $B$ and $C$ in \eqref{p1}, we have that 
 \begin{eqnarray}\label{p2} w_e^2(B,C)&\geq&\frac{w(B^2+C^2)}{2}+\frac{w(C)}{2}\big|w(B+C)-w(B-C)\big|.
\end{eqnarray} 
Therefore, the desired first inequality follows by combining the inequalities in \eqref{p1} and \eqref{p2}.
 Now we prove the second inequality.
 Let $x\in\mathscr{H}$ with $\|x\|=1$.
 \begin{eqnarray*}
&& |\left\langle Bx,x\right\rangle|^2+|\left\langle Cx,x\right\rangle|^2\\
&\leq& \left\langle |B|x,x\right\rangle\left\langle |B^*|x,x\right\rangle+\left\langle |C|x,x\right\rangle\left\langle |C^*|x,x\right\rangle\,\,\,(\textit{using Lemma \ref{lem1}})\\&\leq& \left\lbrace (\left\langle |B|x,x\right\rangle^2+\left\langle |C|x,x\right\rangle^2)(\left\langle |B^*|x,x\right\rangle^2+\left\langle |C^*|x,x\right\rangle^2)\right\rbrace^{\frac{1}{2}}\\
  	&&\,\,\, (\textit{by the inequality $(ab+cd)^2\leq (a^2+c^2)(b^2+d^2)$ for $a,b,c,d\in \mathbb{R}$})\\&=&\left\lbrace |\left\langle |B|x,x\right\rangle+i\left\langle |C|x,x\right\rangle|^2|\left\langle |B^*|x,x\right\rangle+i\left\langle |C^*|x,x\right\rangle|^2\right\rbrace^{\frac{1}{2}}\\&=&\left\lbrace |\left\langle( |B|+i |C|)x,x\right\rangle|^2|\left\langle (|B^*|+i |C^*|)x,x\right\rangle|^2\right\rbrace^{\frac{1}{2}}\\&\leq& w( |B|+i |C|)w(|B^*|+i |C^*|).
 \end{eqnarray*}
 Taking supremum over all $x$ in $\mathscr{H}$ with $\|x\|=1$, we get
\begin{eqnarray}\label{p3}
	w_e^2(B,C)&\leq& w(|B|+i|C|)w(|B^*|+i|C^*|).
\end{eqnarray}
Replacing $C$ by $C^*$ in \eqref{p3}, we have
\begin{eqnarray}\label{p4}
	w_e^2(B,C) &\leq & w(|B|+i|C^*|)w(|B^*|+i|C|).
\end{eqnarray}
Therefore, combining the inequalities in \eqref{p3} and \eqref{p4} we obtain the desired second inequality.
 
\end{proof}

\begin{remark}
(i) Clearly, the lower bound of $w_e(B,C)$ obtained in Theorem \ref{th1} is stronger than the lower bound in \eqref{D06}.	
Also, it is not difficult to verify that $ w^2(|B|+i|C|)\leq\|B^*B+C^*C\|$ and $ w^2(|B^*|+i|C^*|)\leq\|BB^*+CC^*\|.$ Therefore, \begin{eqnarray*}
 w(|B|+i|C|)w(|B^*|+i|C^*|)&\leq&\|B^*B+C^*C\|^{\frac{1}{2}}\|BB^*+CC^*\|^{\frac{1}{2}}.\end{eqnarray*}
Similarly,
\begin{eqnarray*}
	w(|B|+i|C^*|)w(|B^*|+i|C|)&\leq&\|B^*B+CC^*\|^{\frac{1}{2}}\|BB^*+C^*C\|^{\frac{1}{2}}.\end{eqnarray*}
Therefore, it follows from the second inequality in Theorem \ref{th1} that 
\begin{eqnarray*}
w^2_e(B,C) &\leq & \min \left\{ \|B^*B+C^*C\|^{\frac{1}{2}}\|BB^*+CC^*\|^{\frac{1}{2}},  \|B^*B+CC^*\|^{\frac{1}{2}}\|BB^*+C^*C\|^{\frac{1}{2}} \right\}. 
\end{eqnarray*}
The above bound for $w_e(B,C)$ is better than the upper bound in \eqref{eqn1} if $\|BB^*+CC^*\| \leq \|B^*B+C^*C\|.$ \\
(ii) Following Theorem \ref{th1},  $w^2_e(B,C)=\frac12 w(B^2+C^2)$ implies $w(B+C)=w(B-C).$ By considering $C=0$, we conclude that the converse part does not always hold.

 \end{remark}
 
 The following corollary is an immediate consequence of Theorem \ref{th1}  assuming $B$ and $C$ to be  self-adjoint operators.
\begin{cor}
	Let $ B,C\in\mathbb{B}(\mathscr{H})$ be self-adjoint, then
\begin{eqnarray}
\frac{1}{2} \|B^2+C^2\|+\frac{1}{2}\max \{\|B\|,\|C\|\}\big| \|B+C\|-\|B-C\|\big| \nonumber\\ \leq w_e^2(B,C)\leq w^2(|B|+i|C|).
\label{eq2}\end{eqnarray}
\end{cor}

Note that the second inequality in (\ref{eq2}) gives better bound than that in (\ref{eqn2}).
In particular, by considering $ B=\Re(A)$ and $ C=\Im(A)$  in (\ref{eq2}) we obtain the following new upper and lower bounds for the numerical radius of a bounded linear operator $A$.  
\begin{cor} \label{pcor}
	Let $A\in \mathbb{B}(\mathscr{H})$, then
\begin{eqnarray*}
\frac{1}{4} \|A^*A+AA^*\|+\frac{1}{2}\max \{\|\Re(A)\|,\|\Im(A)\|\}\big| \|\Re(A)+\Im(A)\|-\|\Re(A)-\Im(A)\|\big| \\
\leq w^2(A) \leq w^2( |\Re(A)| +i |\Im(A)|).
\end{eqnarray*}
\end{cor}

\begin{remark}
  Clearly, $w^2( |\Re(A)| +i |\Im(A)|)\leq \frac12 \|A^*A+AA^*\|.$ Therefore, the inequality in Corollary \ref{pcor} is stronger then the inequality (\ref{d}). Also, $\frac12 \|A^*A+AA^*\| \leq \| \Re(A) \|^2+\| \Im (A)\|^2$. So, the upper bound for $w(A)$ in Corollary \ref{pcor} is stronger than the well-known bound $w(A)\leq \sqrt{\| \Re(A)\|^2+\| \Im(A) \|^2}.$
\end{remark}

Again, if we consider $B= A$ and $C= A^*$ in Theorem \ref{th1}, we get the following inequality.
\begin{cor}\label{cor1}
	Let $A\in \mathbb{B}(\mathscr{H}),$ then
 	\begin{eqnarray*}
		\frac{1}{2}\|\Re(A^2)\|+\frac{1}{2}w(A)\big| \|\Re(A)\|-\|\Im(A)\|\big| \leq w^2(A) \leq \frac{1}{2}w(|A|+i|A^*|)w(|A^*|+i|A|).
	\end{eqnarray*}
\end{cor}

\begin{remark}
Clearly, the first inequality in Corollary \ref{cor1} is sharper than the existing inequality  $ \frac{1}{2}\|\Re(A^2)\| \leq w^2(A),$ given in \cite[Remark 2]{D}. Observe that  $\frac{1}{2}w(|A|+i|A^*|)w(|A^*|+i|A|) \leq \frac12 \|A^*A+AA^*\|,$	and so 
 the second inequality in Corollary \ref{cor1} is stronger then the second inequality in (\ref{d}).
\end{remark}

Next lower bound for $w_e(B,C)$ reads as follows.

\begin{theorem}
Let $ B,C\in\mathbb{B}(\mathscr{H})$, then $$\frac{1}{2}\max\left\lbrace w^2(B+C)+c^2(B-C),w^2(B-C)+c^2(B+C)\right\rbrace\leq w_e^2(B,C).$$
\label{th2}\end{theorem}
\begin{proof}
Let $x$ be an unit vector in $\mathscr{H}$. Then we have  
\begin{eqnarray*}|\left\langle Bx,x\right\rangle+\left\langle Cx,x\right\rangle|^2+|\left\langle Bx,x\right\rangle-\left\langle Cx,x\right\rangle|^2=2(|\left\langle Bx,x\right\rangle|^2+|\left\langle Cx,x\right\rangle|^2).\end{eqnarray*}Therefore,\begin{eqnarray*}|\left\langle (B+C)x,x\right\rangle|^2+|\left\langle (B-C)x,x\right\rangle|^2&=&2(|\left\langle Bx,x\right\rangle|^2+|\left\langle Cx,x\right\rangle|^2)\\&\leq& 2w_e^2(B,C).
\end{eqnarray*}Thus,\begin{eqnarray*}
|\left\langle (B+C)x,x\right\rangle|^2&\leq&2w_e^2(B,C)-|\left\langle (B-C)x,x\right\rangle|^2\\&\leq&2w_e^2(B,C)-c^2(B-C). 
\end{eqnarray*}
Taking supremum over all $x$ in $\mathscr{H}$ with $\|x\|=1$, we get $$ w^2(B+C)\leq2w_e^2(B,C)-c^2(B-C),$$ that is, \begin{eqnarray} \label{eq3}
	w^2(B+C)+c^2(B-C)\leq2w_e^2(B,C).
\end{eqnarray}
Similarly, we can prove that 
\begin{eqnarray}\label{eq4}
 w^2(B-C)+c^2(B+C)\leq2w_e^2(B,C).
\end{eqnarray} 
Combining the inequalities (\ref{eq3})  and (\ref{eq4}) we get,   $$\frac{1}{2}\max\left\lbrace w^2(B+C)+c^2(B-C),w^2(B-C)+c^2(B+C)\right\rbrace\leq w_e^2(B,C),$$ as desired.
\end{proof}

\begin{remark}
(i) Clearly, the bound in Theorem \ref{th2} is stronger then the first bound in \cite[Th. 2]{D}, that is, $\frac{1}{2}\max\left\lbrace w^2(B+C),w^2(B-C) \right\rbrace\leq w_e^2(B,C)$.\\
(ii) For self-adjoint operators $B,C \in \mathbb{B}(\mathscr{H})$, the bound in Theorem \ref{th2} is of the form $$\frac{1}{2}\max\left\lbrace \|B+C\|^2+c^2(B-C),\|B-C\|^2+c^2(B+C)\right\rbrace\leq w_e^2(B,C).$$
(iii) Replacing $B$ by $\Re(A)$ and $C$ by $\Im(A)$ in  Theorem \ref{th2} we get the following lower bound for the numerical radius of $A\in \mathbb{B}(\mathscr{H})$:
\begin{eqnarray*}
&& w^2(A) \geq \\
&& \frac{1}{2}\max\left\lbrace \|\Re(A)+\Im(A)\|^2+c^2(\Re(A)-\Im(A)),\|\Re(A)-\Im(A)\|^2+c^2(\Re(A)+\Im(A))\right\rbrace. 
\end{eqnarray*}
This bound appeared recently  in  \cite[Cor. 2.3]{PSMK}.
\end{remark}

We next obtain the following inequality.
\begin{theorem}\label{cor2}
Let $ B,C\in\mathbb{B}(\mathscr{H})$, then $$\max\left\lbrace w^2(B)+c^2(C),w^2(C)+c^2(B)\right\rbrace\leq w_e^2(B,C).$$
\end{theorem}
\begin{proof}
Let $x\in \mathscr{H}$ with $\|x\|=1$. Then we have  \begin{eqnarray*}|\left\langle Bx,x\right\rangle+\left\langle Cx,x\right\rangle|^2+|\left\langle Bx,x\right\rangle-\left\langle Cx,x\right\rangle|^2=2(|\left\langle Bx,x\right\rangle|^2+|\left\langle Cx,x\right\rangle|^2),\end{eqnarray*}
that is,
\begin{eqnarray*}|\left\langle (B+C)x,x\right\rangle|^2+|\left\langle (B-C)x,x\right\rangle|^2=2(|\left\langle Bx,x\right\rangle|^2+|\left\langle Cx,x\right\rangle|^2).\end{eqnarray*}
This implies that \begin{eqnarray}
w_e^2(B+C,B-C)=2w_e^2(B,C).
\label{eq5}\end{eqnarray} 
Now replacing $B$ by $B+C$ and $C$ by $B-C$ in  Theorem \ref{th2} we obtain that \begin{eqnarray}
2\max\left\lbrace w^2(B)+c^2(C),w^2(C)+c^2(B)\right\rbrace\leq w_e^2(B+C,B-C).
\label{eq6}\end{eqnarray} 
Therefore, the required inequality follows from (\ref{eq6}) by using the fact (\ref{eq5}).

\end{proof}

\begin{remark}
	(i) Following Theorem \ref{cor2}, we have for $ B,C \in \mathbb{B} (\mathscr{H}),$ $w_e(B,C)=w(B)$ implies if $\lim_{n\to \infty} |\langle Bx_n,x_n\rangle|=w(B)$ then $\lim_{n\to \infty} |\langle Cx_n,x_n\rangle=0.$ It should be mentioned here that the converse part is not necessarily true.\\
(ii) For normal operators $ B,C \in \mathbb{B} (\mathscr{H})$, the inequality in Theorem \ref{cor2} turns into the form  $$\max\left\lbrace \|B\|^2+c^2(C),\|C\|^2+c^2(B)\right\rbrace\leq w_e^2(B,C).$$
(iii) If we replace B by $\Re(A)$ and C by $\Im(A)$ in the inequality in Theorem  \ref{cor2}, then we get the following lower bound (see \cite[Th. 3.3]{BBPLAA}) for the numerical radius of a bounded linear operator $A$ on $\mathscr{H}$: $$\max\left\lbrace \|\Re(A)\|^2+c^2(\Im(A)),\|\Im(A)\|^2+c^2(\Re(A))\right\rbrace\leq w^2(A).$$ 
\end{remark}

To prove the next result we need the following inequality for vectors in $\mathscr{H}$, knows as Buzano's inequality.

\begin{lemma}$($\cite{S}$)$ 
	Let $ x,y,e\in \mathscr{H}$ with $ \|e\|=1$, then
	$$ \mid\langle x,e\rangle\langle e,y\rangle\mid\leq\frac{1}{2}\left(\mid\langle x,y\rangle \mid +\|x\|\|y\|\right).$$\label{lem11}\end{lemma}
\begin{theorem}
Let $ B,C\in\mathbb{B}(\mathscr{H})$, then $$ w_e^2(B,C)\leq \min \left \{ w^2(B+C), w^2(B-C) \right\} +\frac{1}{2}\|C^*C+BB^*\|+w(BC).$$
\label{th3}\end{theorem}
\begin{proof}
Let $x$ be an unit vector in $\mathscr{H}$. Then we have
\begin{eqnarray*}
|\left\langle Cx,x\right\rangle|^2-2Re{[\left\langle Cx,x\right\rangle\overline{\left\langle Bx,x\right\rangle}]}+|\left\langle Bx,x\right\rangle|^2&=&|\left\langle Cx,x\right\rangle-\left\langle Bx,x\right\rangle|^2\\&=&|\left\langle (C- B)x,x\right\rangle|^2\\&\leq& w^2(C-B).
\end{eqnarray*}Thus,\begin{eqnarray*}
|\left\langle Cx,x\right\rangle|^2+|\left\langle Bx,x\right\rangle|^2&\leq& w^2(C-B)+2Re{[\left\langle Cx,x\right\rangle\overline{\left\langle Bx,x\right\rangle}]}\\&\leq&w^2(C-B)+2|\left\langle Cx,x\right\rangle\left\langle Bx,x\right\rangle|\\&\leq&w^2(C-B)+\|Cx\|\|B^*x\|+|\langle Cx,B^*x\rangle|\,(\textit{by Lemma \ref{lem11}})\\&\leq& w^2(C-B)+\frac{1}{2}(\|Cx\|^2+\|B^*x\|^2)+w(BC)\\&\leq& w^2(C-B)+\frac{1}{2}\|C^*C+BB^*\|+w(BC). 
\end{eqnarray*}
Therefore, taking supremum over all $x$ in $\mathscr{H}$ with $\|x\|=1$, we get 
\begin{eqnarray}\label{p6}
	w^2_e(B,C) \leq w^2(B-C)+\frac{1}{2}\|C^*C+BB^*\|+w(BC). 
\end{eqnarray}
Replacing $C$ by $-C$ in the above inequality, we obtain that 
\begin{eqnarray}\label{p7}
	w^2_e(B,C) \leq w^2(B+C)+\frac{1}{2}\|C^*C+BB^*\|+w(BC). 
\end{eqnarray}
Hence, the desired bound follows from \eqref{p6} and \eqref{p7}.
\end{proof}

\begin{remark}
If we take $B=C=A$ in \eqref{p6}, then we get the following upper bound ( see \cite{AK}) for the numerical radius of a bounded linear operator $A$ on $\mathscr{H}$: $$ w^2(A)\leq \frac{1}{4}\|A^*A+AA^*\|+\frac{1}{2}w(A^2).$$ 
\end{remark}

In \cite[Th. 1]{D}, it is proved that if $B,C$ are bounded linear operators on $\mathcal{H},$ then 
 \begin{eqnarray}\label{drag01}
 	w\left(\frac{1}{2} B^2+ \frac{1}{2} C^2\right) \leq w_e^2(B,C).
 \end{eqnarray}
In the next result we establish a generalization of the above  lower bound for the Euclidean operator radius for a pair $(B,C)$ of two bounded linear operators on $\mathscr{H}.$ 
\begin{theorem}\label{th4}
Let $ B,C\in\mathbb{B}(\mathscr{H})$, then 
\begin{eqnarray}\label{pmax}
	\max_{0\leq\alpha\leq1}w\left(\alpha B^2+(1-\alpha)C^2\right) \leq w_e^2(B,C).
	\end{eqnarray}
\end{theorem}

\begin{proof}
Let $x$ be an unit vector in $\mathscr{H}$. Then applying the inequality $(ab+cd)^2\leq (a^2+c^2)(b^2+d^2))$ for real numbers $a,b,c,d$, we have
\begin{eqnarray*}
\sqrt{\alpha}|\langle Bx,x\rangle|+\sqrt{1-\alpha}|\langle Cx,x\rangle|&\leq&(|\langle Bx,x\rangle|^2+|\langle Cx,x\rangle|^2)^{\frac{1}{2}}((\sqrt{\alpha})^2+(\sqrt{1-\alpha})^2)^{\frac{1}{2}}\\
  	&=&(|\langle Bx,x\rangle|^2+|\langle Cx,x\rangle|^2)^{\frac{1}{2}}.
\end{eqnarray*}Therefore,
\begin{eqnarray*}
(|\langle Bx,x\rangle|^2+|\langle Cx,x\rangle|^2)^{\frac{1}{2}}&\geq&\sqrt{\alpha}|\langle Bx,x\rangle|+\sqrt{1-\alpha}|\langle Cx,x\rangle|\\&=&|\langle \sqrt{\alpha}Bx,x\rangle|+|\langle\sqrt{1-\alpha} Cx,x\rangle|\\&\geq&|\langle \sqrt{\alpha}Bx,x\rangle\pm\langle\sqrt{1-\alpha} Cx,x\rangle|\\&=&\langle\left( \sqrt{\alpha}B\pm\sqrt{1-\alpha} C\right) x,x\rangle|.
\end{eqnarray*}
Taking supremum over all $x$ in $\mathscr{H}$ with $\|x\|=1$, we get $$ w_e(B,C)\geq w(\sqrt{\alpha}B\pm\sqrt{1-\alpha}C).$$ 
Therefore,
\begin{eqnarray*}
2 w_e^2(B,C)&\geq& w^2(\sqrt{\alpha}B+\sqrt{1-\alpha}C)+ w^2(\sqrt{\alpha}B-\sqrt{1-\alpha}C)\\&\geq& w((\sqrt{\alpha}B+\sqrt{1-\alpha}C)^2)+ w((\sqrt{\alpha}B-\sqrt{1-\alpha}C)^2)\\&\geq& w((\sqrt{\alpha}B+\sqrt{1-\alpha}C)^2+(\sqrt{\alpha}B-\sqrt{1-\alpha}C)^2)\\&=& 2w(\alpha B^2+(1-\alpha)C^2).
\end{eqnarray*}
This implies $w_e^2(B,C)\geq w(\alpha B^2+(1-\alpha)C^2).$ 
This holds for all $\alpha \in [0,1]$ and so we get the desired inequality.
\end{proof}

\begin{remark}
(i) Clearly the inequality (\ref{pmax}) is a refinement of the inequality (\ref{drag01}) obtained in \cite[Th. 1]{D}. To see that the refinement is proper,  consider $B=\begin{pmatrix}
	1&0\\
	0&0
\end{pmatrix}$ and $C=\begin{pmatrix}
0&0\\
0&2
\end{pmatrix}.$ Then we  get,  $$w\left(\frac{1}{2} B^2+ \frac{1}{2} C^2\right)=2< 4=\max_{0\leq\alpha\leq1} w(\alpha B^2+(1-\alpha)C^2).$$


(ii)  If we replace $B$ by $\Re(A)$ and $C$ by $\Im(A)$ in Theorem \ref{th4}, we obtain that for $ A\in\mathbb{B}(\mathscr{H}),$
 \begin{eqnarray}\label{corn1}
 \left\|\alpha (\Re(A))^2+(1-\alpha)(\Im(A))^2\right\|\leq w^2(A),
 \end{eqnarray}
for all $\alpha \in [0,1].$
  In particular, for $\alpha=\frac12,$  we get the well-known lower bound
  $$\frac14\| A^*A+AA^*\|\leq w^2(A).$$
\end{remark}  
  
  Next, we prove the following theorem.

\begin{theorem}\label{th5}
Let $ B,C\in\mathbb{B}(\mathscr{H})$, then 
$$ w_e^2(B,C)\leq w^2(\sqrt{\alpha} B+\sqrt{1-\alpha}C)+w^2(\sqrt{1-\alpha} B+\sqrt{\alpha}C),$$ for all $\alpha\in [0,1].$ 
\end{theorem}

\begin{proof}
Let $x$ be an unit vector in $\mathscr{H}$. Then we have
\begin{eqnarray*}
|\langle Bx,x\rangle|^2+|\langle Cx,x\rangle|^2&=&|\langle \sqrt{\alpha}Bx,x\rangle+\langle\sqrt{1-\alpha}Cx,x\rangle|^2+|\langle \sqrt{1-\alpha}Bx,x\rangle-\langle\sqrt{\alpha}Cx,x\rangle|^2\\&=&|\langle (\sqrt{\alpha}B+\sqrt{1-\alpha}C)x,x\rangle|^2+|\langle( \sqrt{1-\alpha}B-\sqrt{\alpha}C)x,x\rangle|^2\\&\leq& w^2(\sqrt{\alpha}B+\sqrt{1-\alpha}C)+w^2( \sqrt{1-\alpha}B-\sqrt{\alpha}C).
\end{eqnarray*}
Therefore, taking supremum over all $x$ in $\mathscr{H}$ with $\|x\|=1$, we get $$ w_e^2(B,C)\leq w^2(\sqrt{\alpha} B+\sqrt{1-\alpha}C)+w^2(\sqrt{1-\alpha} B-\sqrt{\alpha}C),$$ as desired.
\end{proof}
\begin{remark}
(i)  In particular, if we consider $\alpha=\frac12$ in Theorem \ref{th5}, then we get the following upper bound (see \cite[Th. 2]{D}) for Euclidean operator radius 
$$w_e^2(B,C)\leq \frac12 \left( w^2(B+C)+ w^2(B-C)\right).$$
 (ii) Putting $B=\Re(A)$ and $C=\Im(A)$ in Theorem \ref{th5} we obtain the following upper bound for the numerical radius of  a bounded linear operator $A$ on $\mathscr{H}$, $$ w^2(A)\leq \|\sqrt{\alpha} \Re(A)+\sqrt{1-\alpha} \Im(A)\|^2+\|\sqrt{1-\alpha} \Re(A)-\sqrt{\alpha}\Im(A)\|^2,$$ for all $\alpha \in [0,1]$.
\end{remark}

To prove the next upper bound we need the following lemma known as Jensen's inequality, obtained from more general result for superquadratic functions \cite{b1}.

\begin{lemma} The following inequality  
$$\left(\frac{1}{n}\sum_{k=1}^{n}{a_k}\right)^p\leq\frac{1}{n}\sum_{k=1}^{n}{a_k^p}-\frac{1}{n}\sum_{k=1}^{n}\big|a_k-\frac{1}{n}\sum_{j=1}^{n}{a_j}\big|^p,$$ holds for $p\geq2$ and for every finite positive sequence of real numbers $a_1,a_2,\ldots,a_n$.
\label{b3}\end{lemma}
\begin{theorem}
Let $ B,C\in\mathbb{B}(\mathscr{H})$ , then $$w_e^{2r}(B,C)\leq\frac{1}{2}w^{2r}(B+C)+\frac{1}{2}w^{2r}(B-C)-2^r\inf_{\|x\|=1}|Re(\langle Bx,x\rangle\overline{\langle Cx,x\rangle})|$$ holds for every $r\geq2$.
\end{theorem}
\begin{proof}
Let $x$ be an unit vector in $\mathscr{H}$. Then
\begin{eqnarray*}
&&\left(|\langle Bx,x\rangle|^2+|\langle Cx,x\rangle|^2\right)^r\\&=&\left(\frac{1}{2}|\left\langle Bx,x\right\rangle+\left\langle Cx,x\right\rangle|^2+\frac{1}{2}|\left\langle Bx,x\right\rangle-\left\langle Cx,x\right\rangle|^2\right)^r\\&=&\left(\frac{1}{2}|\left\langle (B+ C)x,x\right\rangle|^2+\frac{1}{2}|\left\langle (B- C)x,x\right\rangle|^2\right)^r\\&\leq&\frac{1}{2}|\left\langle (B+ C)x,x\right\rangle|^{2r}+\frac{1}{2}|\left\langle (B- C)x,x\right\rangle|^{2r}\\
&&-\frac{1}{2}\big|\frac{1}{2}\left|\langle (B+ C)x,x\right\rangle|^2-\frac{1}{2}|\left\langle (B- C)x,x\right\rangle|^2\big|^r\\&&-\frac{1}{2}\big|\frac{1}{2}|\left\langle (B- C)x,x\right\rangle|^2-\frac{1}{2}|\left\langle (B+ C)x,x\right\rangle|^2\big|^r\,\,(\textit{using Lemma \ref{b3}})\\&=&\frac{1}{2}|\left\langle (B+ C)x,x\right\rangle|^{2r}+\frac{1}{2}|\left\langle (B- C)x,x\right\rangle|^{2r}\\
&& -\frac{1}{2^{r}}\big||\left\langle (B+ C)x,x\right\rangle|^2-|\left\langle (B- C)x,x\right\rangle|^2\big|^r\\&=&\frac{1}{2}|\left\langle (B+ C)x,x\right\rangle|^{2r}+\frac{1}{2}|\left\langle (B- C)x,x\right\rangle|^{2r}-\frac{2^{2r}}{2^{r}}|Re(\langle Bx,x\rangle\overline{\langle Cx,x\rangle})|\\&\leq&\frac{1}{2}w^{2r}(B+C)+\frac{1}{2}w^{2r}(B-C)-2^r\inf_{\|x\|=1}|Re(\langle Bx,x\rangle\overline{\langle Cx,x\rangle})|.
\end{eqnarray*}
Therefore, taking supremum over all $x$ in $\mathscr{H}$ with $\|x\|=1$,  we get the desire result.
\end{proof}

Next we need the following lemma, called generalized Cauchy-Schwarz inequality \cite{a1}.
\begin{lemma}
If $f$ and $g$ be two non-negative continuous functions on $[0,\infty)$ satisfying $f(t)g(t)=t$ for all $t\in[0,\infty),$ then $$ |\langle Ax,y\rangle| \leq \|f(|A|)x\|\|g(|A^*|y\|,$$ for all $ A\in\mathbb{B}(\mathscr{H})$ and $x,y\in\mathscr{H}.$
\label{b2}\end{lemma}
Finally, we obtain the following inequality involving non-negative continuous functions.
\begin{theorem}
Let $ B,C\in\mathbb{B}(\mathscr{H})$ and let $f,g$ be two non-negative continuous functions on $[0,\infty)$ satisfying $f(t)g(t)=t$ for all $t\in[0,\infty),$  then $$\frac{1}{2}\|B+C\|^2\leq w_e(f^2(|B|),f^2(|C|))w_e(g^2(|B^*|),g^2(|C^*|)).$$

In particular, $$\frac{1}{2}\|B+C\|^2\leq w_e(|B|,|C|)w_e(|B^*|,|C^*|).$$
\end{theorem}
\begin{proof}
Let $x,y$ be two unit vectors in $\mathscr{H}$. Then
\begin{eqnarray*}
&& |\langle(B+C)x,y\rangle|^2\\ &=&|\langle Bx,y\rangle+\langle Cx,y\rangle|^2\\&\leq& 2(|\langle Bx,y\rangle|^2+|\langle Cx,y\rangle|^2)\\&\leq& 2(\|f(|B|)x\|^2\|g(|B^*|)y\|^2+\|f(|C|)x\|^2\|g(|C^*|)y\|^2)\,\,(\textit{using Lemma \ref{b2} })\\&=& 2(\langle f^2(|B|)x,x\rangle\langle g^2(|B^*|)y,y\rangle+\langle f^2(|C|)x,x\rangle\langle g^2(|C^*|)y,y\rangle)\\&\leq& 2\left(\langle f^2(|B|)x,x\rangle^2+\langle f^2(|C|)x,x\rangle^2\right)^\frac{1}{2}\left(\langle g^2(|B^*|)x,x\rangle^2+\langle g^2(|C^*|)x,x\rangle^2\right)^\frac{1}{2}\\&\leq& 2 w_e(f^2(|B|),f^2(|C|))w_e(g^2(|B^*|),g^2(|C^*|)).
\end{eqnarray*}
Taking supremum over $\|x\|=\|y\|=1$, we get $$\frac{1}{2}\|B+C\|^2\leq  w_e(f^2(|B|),f^2(|C|))w_e(g^2(|B^*|),g^2(|C^*|)).$$
\end{proof}

In particular, if we take $f(t)=g(t)=t^\frac{1}{2}$, then  $$\frac{1}{2}\|B+C\|^2\leq w_e(|B|,|C|)w_e(|B^*|,|C^*|),$$ as desired.

\bibliographystyle{amsplain}

\end{document}